\documentclass[a4,11pt, leqno]{article}
\usepackage{amsmath,amsfonts,amsthm,latexsym,amssymb,mathrsfs}
\usepackage{amsmath}
\usepackage{amssymb}
\usepackage{graphicx}
\usepackage{amsthm}
\usepackage{amsfonts}
\usepackage{amscd}

\newtheorem{theorem}{Theorem}[section]
\newtheorem{definition}[theorem]{Definition}
\newtheorem{lemma}[theorem]{Lemma}

\newtheorem{example}[theorem]{Example}
\newtheorem{remark}[theorem]{Remark}
\newtheorem{corollary}[theorem]{Corollary}

\def\inter{{\rm int}}
\def\iso{{\rm iso}}
\def\acc{{\rm acc}}

\def\snoi{\smallskip\noindent}

\def\ind{{\rm ind}}

\date{}

\title{On Koliha-Drazin invertible operators and Browder type theorems}

\author{Milo\v s D. Cvetkovi\'c, Sne\v zana \v C. \v Zivkovi\'c-Zlatanovi\'c \footnote{Sne\v zana \v C. \v Zivkovi\'c-Zlatanovi\'c is supported by the Ministry of Education, Science and Technological Development, Republic of Serbia, grant no. 174007.}}

\begin{document}
	\maketitle
	
\begin{abstract}
\noindent Let $T$ be a bounded linear operator on a Banach space $X$. We give new necessary and sufficient conditions for $T$ to be Drazin or Koliha-Drazin invertible. All those conditions have the following form: $T$ possesses certain decomposition property and zero is not an interior point of some part of the spectrum of $T$. In addition, we study operators $T$ satisfying Browder\textquoteright s theorem, or a-Browder\textquoteright s theorem, by means of some relationships between diferent parts of the spectrum of $T$.
\end{abstract}



\section{Introduction and preliminaries}

Throughout, $X$ is an infinite dimensional complex Banach space, $T$ is a bounded linear operator acting on $X$, and $L(X)$ is the algebra of all bounded linear operators defined on $X$. A subspace $M$ of $X$ is said to be $T$-invariant if $T(M) \subset M$. We define $T_M:M \to M$ as
$T_Mx=Tx, \, x \in M$. Clearly, $T_M$ is linear and bounded. If $M$ and $N$ are two closed $T$-invariant
subspaces of $X$ such that $X=M \oplus N$, we say that $T$ is
completely reduced by the pair $(M,N)$ and it is denoted by
$(M,N) \in Red(T)$. In this case we write $T=T_M \oplus T_N$ and we say
that $T$ is the direct sum of $T_M$ and $T_N$.

There are several well known characterizations of Koliha-Drazin invertible operators \cite{koliha, schmoeger}. For instance, $T \in L(X)$ is Koliha-Drazin invertible if and only if $T$ admits the following decomposition: $T=T_1 \oplus T_2$, $T_1$ is invertible and $T_2$ is quasinilpotent \cite[Theorem 7.1]{koliha}. In this note we give new characterizations that involve some decomposition properties of $T$ as well (see Theorem \ref{Tequiv}). In particular, we prove that $T \in L(X)$ is Koliha-Drazin invertible if and only if $0$ is not an interior point of its descent spectrum and $T=T_1 \oplus T_2$ where $T_1$ is upper semi-Weyl and $T_2$ is quasinilpotent. Using a similar approach we also consider Drazin invertible operators (see Theorem \ref{Tequiv1}).

We will denote by $\sigma(T)$ the spectrum of $T$. Let R be a subset of $L(X)$ and let $\sigma_R(T)=\{\lambda \in \mathbb{C}: T-\lambda I \not \in R\}$ be the spectrum generated by R. We are also interested in the following question: under what conditions is it true that $0 \in \iso \, \sigma_R(T) \Longleftrightarrow 0 \in \iso \, \sigma(T)$, where $\iso \, \sigma_R(T)$ and $\iso \, \sigma(T)$ are the sets of isolated points of $\sigma_R(T)$ and $\sigma(T)$, respectively? We consider the case where R is the set of lower semi-Weyl operators (see Corollary \ref{cor-iso}). Recently, P. Aiena and S. Triolo studied the case where R is the set of surjective operators \cite[Theorem 2.4]{Salvatore}. 

What is more, in Section 3 we study Browder type theorems. In particular, we extend \cite[Theorem 2.3]{aienaTstar} in a sense that we give conditions under which the assertions (iii) and (iv) of \cite[Theorem 2.3]{aienaTstar} can be reversed (see (ii) $\Longleftrightarrow$ (iv) of Theorems \ref{cor-aBrowder} and \ref{converseprime}). 

In our work we will frequently use Theorem \ref{Theorem_1} which enable us to study both Kolih-Drazin invertible operators and Browder-type theorems. In addition, we give several examples that serve to illustrate our results.

In what follows we will recall some necessary facts and give the auxiliary results. Let $N(T)$ and $R(T)$ be the null space and range of $T$, respectively. Denote by $\alpha(T)$ and $\beta(T)$, the dimension of $N(T)$ and the codimension of $R(T)$, respectively.
\begin{definition}
	{\em Let $T \in L(X)$. Then:
	
	\smallskip
	
	\noindent {\rm (i)} $T$ is Kato if $R(T)$ is closed and $N(T) \subset R(T^n)$ for all $n \in \mathbb{N}_0$;
	
	\smallskip
	
	\noindent {\rm (ii)} $T$ is of Kato type if there exists a pair $(M,N) \in Red(T)$ such that $T_M$ is Kato and $T_N$ is nilpotent;
	
	\smallskip
	
	\noindent {\rm (iii)} $T$ admits a generalized Kato decomposition (GKD for short) if there exists a pair $(M,N) \in Red(T)$ such that $T_M$ is Kato and $T_N$ is quasinilpotent.}
\end{definition} 

\noindent If $T$ is Kato, then $T^n$ is Kato for all $n \in \mathbb{N}$ \cite[Theorem 12.7]{Muller}, and hence $R(T^n)$ is closed. We define $k(T)$, the lower bound of $T$, to be $k(T)=\inf \{\|Tx\|:x \in X \, \text{with} \, \|x\|=1\}$. We say that $T$ is bounded below if $k(T)>0$. It may be shown that $T \in L(X)$ is bounded below if and only if $T$ is injective and $R(T)$ is closed. The approximate point spectrum of $T \in L(X)$, denoted by $\sigma_{ap}(T)$, is the set of all $\lambda \in \mathbb{C}$ such that the operator $T-\lambda I$ is not bounded below. The surjective spectrum $\sigma_{su}(T)$ is defined as the set of all $\lambda \in \mathbb{C}$ such that $R(T-\lambda I) \neq X$. It is well known that $\sigma_{ap}(T)$ and $\sigma_{su}(T)$ are compact subsets of $\mathbb{C}$ that contain the boundary of $\sigma(T)$. Evidently, bounded below operators and surjective operators are examples of Kato operators.

We say that $T \in L(X)$ is semi-Fredholm if $R(T)$ is closed and either $\alpha(T)<\infty$ or $\beta(T)<\infty$. In such a case we define the index of $T$ as $\ind(T)=\alpha(T)-\beta(T)$. It is well known that every semi-Fredholm operator is of Kato type with $N$ finite dimensional \cite[Theorem 16.21]{Muller}. An operator $T \in L(X)$ is said to be upper semi-Fredholm (resp. lower semi-Fredholm, Fredholm) if $R(T)$ is closed and $\alpha(T)<\infty$ (resp. $\beta(T)<\infty$, $\alpha(T)<\infty$ and $\beta(T)<\infty$).

Other important classes of operators in Fredholm theory are the classes of upper and lower semi-Weyl operators. These classes are defined as follows: 
$T \in L(X)$ is said to be upper semi-Weyl (resp. lower semi-Weyl) if it is upper semi-Fredholm and $\ind(T) \leq 0$ (resp. lower semi-Fredholm and $\ind(T) \geq 0$).  The upper semi-Weyl spectrum is the set
$$\sigma_{uw}(T)=\{\lambda \in \mathbb{C}: T-\lambda I \; \text{is not upper semi-Weyl}\},$$
while the lower semi-Weyl spectrum is defined by
$$\sigma_{lw}(T)=\{\lambda \in \mathbb{C}: T-\lambda I \; \text{is not lower semi-Weyl}\}.$$
It is said that $T \in L(X)$ is Weyl if it is Fredholm of index $0$. The Weyl spectrum is defined by
$$\sigma_{w}(T)=\{\lambda \in \mathbb{C}: T-\lambda I \; \text{is not Weyl}\}.$$

Recall that the ascent of an operator $T \in L(X)$ is defined as the smallest nonnegative integer $p=p(T)$ such that $N(T^p)=N(T^{p+1})$. If such an integer does not exist, then $p(T)=\infty$. Similarly, the descent of $T$ is defined as the smallest nonnegative integer $q=q(T)$ such that $R(T^{q})=R(T^{q+1})$, and if such an integer does not exist, we put  $q(T)=\infty$. The descent spectrum is the set 
$$\sigma_{desc}(T)=\{\lambda \in \mathbb{C}: q(T-\lambda I)=\infty\}.$$
The set $\sigma_{desc}(T)$ is compact but possibly empty (see Corollary 1.3 and Theorem 1.5 in \cite{descent}). For example, if $T$ is the zero operator on $X$, then $\sigma_{desc}(T)=\emptyset$.

An operator $T \in L(X)$ is upper semi-Browder if $T$ is upper semi-Fredholm and $p(T)<\infty$, while an operator $T \in L(X)$ is Browder if it is Fredholm of finite ascent and descent. We consider the following spectra generated by these classes:
\begin{center}
	$\sigma_{ub}(T)=\{\lambda \in \mathbb{C}: T-\lambda I \; \text{is not upper semi-Browder}\}$,
	
	\smallskip
	
	$\sigma_{b}(T)=\{\lambda \in \mathbb{C}: T-\lambda I \; \text{is not Browder}\}$.
\end{center}
The sets $\sigma_{ub}(T)$ and $\sigma_b(T)$ are the upper-Browder spectrum and the Browder spectrum of $T$, respectively.

Drazin invertibility was initially introduced in rings \cite{drazin}. In the case of Banach spaces, an operator $T \in L(X)$ is Drazin invertible if and only if $p(T)=q(T)<\infty$ \cite[Theorem 4]{king}, and this is exactly when $T=T_1 \oplus T_2$, where $T_1$ is invertible and $T_2$ is nilpotent \cite[Theorem 22.10]{Muller}. The concept of Drazin invertibility for bounded operators may be extended as follows. 

\begin{definition}
	{\em $T \in L(X)$ is said to be left Drazin invertible if $p=p(T)<\infty$ and $R(T^{p+1})$ is closed; $T \in L(X)$ is said to be right Drazin invertible if $q=q(T)<\infty$ and $R(T^{q})$ is closed.}
\end{definition}
\noindent The Drazin spectrum, the left Drazin spectrum and the right Drazin spectrum are defined respectively by
	$$\sigma_{D}(T)=\{\lambda \in \mathbb{C}: T-\lambda I \; \text{is not Drazin invertible}\},$$
	$$\sigma_{LD}(T)=\{\lambda \in \mathbb{C}: T-\lambda I \; \text{is not left Drazin invertible}\},$$
	$$\sigma_{RD}(T)=\{\lambda \in \mathbb{C}: T-\lambda I \; \text{is not right Drazin invertible}\}.$$

\noindent In particular, we note that the left Drazin spectrum is a closed subset of the complex plane \cite{axiomatic}. J. Koliha also extended the concept of Drazin invertibility \cite{koliha}. An operator $T \in L(X)$ is generalized Drazin invertible (or Koliha-Drazin invertible) if and only if zero is not an accumulation point of its spectrum, and this is equivalent to saying that $T=T_1 \oplus T_2$, where $T_1$ is invertible and $T_2$ is quasinilpotent \cite[Theorems 4.2 and 7.1]{koliha}.

From now on, if $F \subset \mathbb{C}$ then the set of isolated points of $F$, the set of accumulation points of $F$, the boundary of $F$ and the interior of $F$ will be denoted by $\iso \, F$, $\acc \, F$, $\partial F$ and $\inter \, F$, respectively.

\begin{lemma} \label{partialcluster}
	Let $F \subset \mathbb{C}$ be closed. Then
	$\lambda \in \iso \, F$ if and only if $\lambda \in \iso \, \partial F$.
\end{lemma}

\begin{lemma} \label{opsta}
	Let $E$ and $F$ be compact sets of the complex plane such that
	$\partial F \subset E \subset F$. If $\lambda_0 \in \partial F$,
	then the following statements are equivalent:\par
	
	\smallskip
	
	\noindent {\rm (i)} $\lambda_0 \in \iso \, E$;\par
	
	\smallskip
	
	\noindent {\rm (ii)} $\lambda_0 \in \iso \, F$.
\end{lemma}
\begin{proof}
	If $\lambda_0 \in \iso \, F$, then $\lambda_0 \in \partial F \subset
	E$. It follows that $\lambda_0 \in \iso \, E$ since $E \subset F$.
	To prove the opposite implication suppose that $\lambda_0 \in
	\iso \, E \cap \partial F$. Evidently, $\lambda_0 \in \iso \, \partial F$. Now, we obtain $\lambda_0 \in \iso \, F$ by Lemma \ref{partialcluster}.
\end{proof}

\smallskip

The quasinilpotent part of an operator $T \in L(X)$ is defined by
$$H_0(T)=\{x \in X: \lim_{n \to \infty} \|T^nx\|^{1/n}=0\}.$$
Clearly, $H_0(T)$ is a subspace of $X$ not necessarily closed. An operator $T \in L(X)$ has the single-valued extension property at $\lambda_0 \in \mathbb{C}$, SVEP at $\lambda_0$, if for every open disc $D_{\lambda_0}$ centered at $\lambda_0$ the only analytic function $f: D_{\lambda_0} \to X$ which satisfies
$$(T-\lambda I)f(\lambda)=0 \; \; \text{for all} \; \; \lambda \in D_{\lambda_0},$$
is the function $f \equiv 0$. The operator $T$ is said to have the SVEP if $T$ has the SVEP at every $\lambda \in \mathbb{C}$. The following implication follows immediately from the definition of the localized SVEP and from the identity theorem for analytical functions:
\begin{equation}\label{SVEPap}
\lambda_0 \not \in \inter \, \sigma_{ap}(T) \Longrightarrow T \; \text{has the SVEP at} \; \lambda_0.
\end{equation} 
Consequently,
\begin{equation}\label{SVEPsu}
\lambda_0 \not \in \inter \, \sigma_{su}(T) \Longrightarrow T^{\prime} \; \text{has the SVEP at} \; \lambda_0,
\end{equation}
since $\sigma_{ap}(T^{\prime})=\sigma_{su}(T)$, where $T^{\prime}$ is the adjoint operator of $T$. In particular, from \eqref{SVEPap} and \eqref{SVEPsu} we easily obtain that
\begin{equation*}
\lambda_0 \not \in \acc \, \sigma(T) \Longrightarrow \; \text{both} \; T \; \text{and} \; T^{\prime} \; \text{has the SVEP at} \; \lambda_0.
\end{equation*}
Under the assumption that $T$ admits a GKD, we have the following result (see \cite[Theorem 3.5]{kinezi} and \cite[Theorem 2.6]{aienaAmer}).
\begin{lemma}\label{lema1}
Suppose that $T \in L(X)$ admits a GKD. Then the following properties are equivalent:

\smallskip

\noindent {\rm (i)} $H_0(T)$ is closed;

\smallskip

\noindent {\rm (ii)} $T$ has the {\rm SVEP} at $0$;

\smallskip

\noindent {\rm (iii)} $0$ is not an interior point of $\sigma_{ap}(T)$.	
\end{lemma}

\noindent For more comprehensive study of the SVEP we refer the reader to \cite{aienabook} and \cite{LN}.

\section{Koliha-Drazin invertible operators}

First, we are concerned with the relationships between different parts of the spectrum.

\begin{theorem}\label{Theorem_1}
	Let $T \in L(X)$. Then
	
	\snoi {\rm (i)} $\iso \, \sigma_{ap}(T) \subset \iso \, \sigma(T) \cup \inter \, \sigma_{desc}(T)$;
	
	\snoi {\rm (ii)} $\iso \, \sigma_{su}(T) \subset \iso \, \sigma(T) \cup \inter \, \sigma_{LD}(T)$;
	
	\snoi {\rm (iii)} $\iso \, \sigma_{uw}(T) \subset \iso \, \sigma_b(T) \cup \inter \, \sigma_{desc}(T)$;
	
	\snoi {\rm (iv)} $\iso \, \sigma_{lw}(T) \subset \iso \, \sigma_b(T) \cup \inter \, \sigma_{LD}(T)$;
	
	\smallskip
	
	\noindent {\rm (v)} $\sigma(T)=\sigma_{ap}(T) \cup \inter \, \sigma_{desc}(T)$;
	
	\smallskip
	
	\noindent {\rm (vi)} $\sigma(T)=\sigma_{su}(T) \cup \inter \, \sigma_{LD}(T)$;
	
	\smallskip
	
	\noindent {\rm (vii)} $\sigma_b(T)=\sigma_{uw}(T) \cup \inter \, \sigma_{desc}(T)$;

	\smallskip
	
	\noindent {\rm (viii)} $\sigma_b(T)=\sigma_{lw}(T) \cup \inter \, \sigma_{LD}(T)$. 
	
	
\end{theorem}
\begin{proof}
	{\rm (i)}.	Let $\lambda_0 \in \iso \, \sigma_{ap}(T)\setminus\inter \, \sigma_{desc}(T)$. Then there exists a sequence $(\lambda_n)$ converging to $\lambda_0$ such that $T-\lambda_n I$ is bounded below and $q(T-\lambda_n I)$ is finite for all $n \in \mathbb{N}$. From \cite[Theorem 3.4, part (ii)]{aienabook} it follows that $T-\lambda_nI$ is invertible for every $n \in \mathbb{N}$ and that $\lambda_0$ is a boundary point of $\sigma(T)$. Now, applying Lemma \ref{opsta} with $F=\sigma(T)$ and $E=\sigma_{ap}(T)$ gives $\lambda_0 \in \iso \, \sigma(T)$.
	
	\smallskip
	
	\noindent {\rm (ii)}. Use $\partial \sigma(T) \subset \sigma_{su}(T) \subset \sigma(T)$ and \cite[Theorem 3.4, part (i)]{aienabook}, and proceed analogously to the proof of {\rm (i)}.
	
	\smallskip
	
	\noindent {\rm (iii)}. Let $\lambda_0 \in \iso \, \sigma_{uw}(T) \setminus \inter \, \sigma_{desc}(T)$. Then there exists a sequence $(\lambda_n)$, $\lambda_n \to \lambda_0$ as $n \to \infty$, such that $q(T-\lambda_nI)$ is finite and  $T-\lambda_nI$ is upper semi-Weyl for all $n \in \mathbb{N}$. Consequently, $\alpha(T-\lambda_nI) \leq \beta(T-\lambda_nI)$, $n \in \mathbb{N}$. Further, we have  $\beta(T-\lambda_nI) \leq \alpha(T-\lambda_nI)$ by \cite[Theorem 3.4, part (ii)]{aienabook}, so $\alpha(T-\lambda_nI)=\beta(T-\lambda_nI)<\infty$. We conclude from \cite[Theorem 3.4, part (iv)]{aienabook} that $p(T-\lambda_nI)$ is finite, hence that $T-\lambda_n I$ is Browder, and finally that $\lambda_0 \in \partial \sigma_b(T)$. Since $\partial \sigma_b(T) \subset \sigma_{uw}(T) \subset \sigma_b(T)$ \cite[Corollary 2.5]{voki}, $\lambda_0 \in \iso \, \sigma_b(T)$ by Lemma \ref{opsta}.
	
	\smallskip
	
	\noindent {\rm (iv)}. Using \cite[Corollary 2.5]{voki} and a standard duality argument we obtain $\partial \sigma_b(T) \subset \sigma_{lw}(T) \subset \sigma_b(T)$. Now, the result follows by the same method as in {\rm (iii)}.
	
	\smallskip
	
	\noindent {\rm (v).} Since $\sigma_{ap}(T) \cup \inter \, \sigma_{desc}(T) \subset \sigma(T)$ is clear, it suffices to show the opposite inclusion. Let $\lambda_0  \in \sigma(T)$ and $\lambda_0 \not \in \sigma_{ap}(T) \cup \inter \, \sigma_{desc}(T)$.  As in the proof of (i), we conclude $\lambda_0 \in \partial \sigma(T)$, and hence $\lambda_0 \in\sigma_{ap}(T)$, which is a  contradiction.
	
	\smallskip 
	
	\noindent {\rm (vii).} The inclusion  $\sigma_{uw}(T) \cup \inter \, \sigma_{desc}(T) \subset \sigma_{b}(T)$ is clear. Let $\lambda_0\in \sigma_{b}(T)$ and $\lambda_0 \not \in \sigma_{uw}(T) \cup \inter \, \sigma_{desc}(T)$. Then analysis similar to that in the proof of (iii) shows that $\lambda_0 \in \partial \sigma_b(T)$. This is impossible since $\partial \sigma_b(T) \subset \sigma_{uw}(T)$.
	
	\smallskip
	
	\noindent {\rm (vi)} and {\rm (viii).} Apply an argument similar to the one used in the proof of statements {\rm (v)} and {\rm (vii)}, respectively.
\end{proof} 


In the following result we give further characterizations of Koliha-Drazin invertible operators.

\begin{theorem}\label{Tequiv}
	Let $T \in L(X)$. The following conditions are equivalent:
	
	\smallskip
	
	\noindent {\rm (i)} $T$ is Koliha-Drazin invertible;
	
    \smallskip
	
	\noindent {\rm (ii)} $0 \not \in \inter \, \sigma_{desc}(T)$ and there exists $(M,N) \in Red(T)$ such that $T_M$ is upper semi-Weyl and $T_N$ is quasinilpotent;
	
	\smallskip
	
	\noindent {\rm (iii)} $0 \not \in \inter \, \sigma_{LD}(T)$ and there exists $(M,N) \in Red(T)$ such that $T_M$ is lower semi-Weyl and $T_N$ is quasinilpotent.
\end{theorem}
\begin{proof}
	{\rm (i)} $\Longrightarrow$ {\rm (ii)}. There exists a pair $(M,N) \in Red(T)$ such that $T_M$ is invertible and $T_N$ is quasinilpotent. Clearly, $T_M$ is upper semi-Weyl and $0 \not \in \acc \, \sigma(T)$.	Since $\inter \, \sigma_{desc}(T) \subset \inter \, \sigma(T) \subset \acc \, \sigma(T)$ then $0 \not \in \inter \, \sigma_{desc}(T)$.
	
	\smallskip
	
	\noindent {\rm (ii)} $\Longrightarrow$ {\rm (i)}. According to \cite[Theorem 3.4]{caot}, $T$ admits a {\rm GKD} and $0 \not \in \acc \, \sigma_{uw}(T)$. We will prove that $0 \not \in \acc \, \sigma_{b}(T)$. If $0 \in \iso \, \sigma_{uw}(T)$, then Theorem \ref{Theorem_1}, part {\rm (iii)}, implies $0 \in \iso \, \sigma_{b}(T)$. If $0 \not \in \sigma_{uw}(T)$, then  $0 \not \in \sigma_{b}(T)$ by Theorem \ref{Theorem_1}, part {\rm (vii)}. Applying \cite[Theorem 3.9]{caot} we obtain the desired conclusion.
	
	\smallskip
	
	\noindent {\rm (i)} $\Longleftrightarrow$ {\rm (iii)}. Apply an argument similar to the one used in the proof of the equivalence {\rm (i)} $\Longleftrightarrow$ {\rm (ii)}, and use parts (iv) and (viii) of Theorem \ref{Theorem_1}.
\end{proof}

Recently, P. Aiena and S. Triolo proved that if $H_0(T)$ is closed then $0 \in \iso \, \sigma_{su}(T)$ if and only if $0 \in \iso \, \sigma(T)$ \cite[Theorem 2.4]{Salvatore}. In what follows we show that under an additional assumption this assertion remains valid if the surjective spectrum is replaced by the lower Weyl spectrum of $T$.

\begin{corollary}\label{cor-iso}
	Let $T \in L(X)$ such that $0 \in \sigma_{lw}(T)$. Suppose that $T$ admits a GKD and that $H_0(T)$ is closed. Then:
	
	\smallskip
	
	$$0 \in \iso \, \sigma_{lw}(T) \Longleftrightarrow 0 \in \iso \, \sigma(T).$$	
\end{corollary}
\begin{proof}
	If $0 \in \iso \, \sigma(T)$ it follows immediately that $0 \not \in \acc \, \sigma_{lw}(T)$. Consequently, $0 \in \iso \, \sigma_{lw}(T)$.
	
	Suppose that $0 \in \iso \, \sigma_{lw}(T)$. From \cite[Theorem 3.4]{caot} we see that $T$ can be represented as $T=T_M \oplus T_N$ with $T_M$ lower semi-Weyl and $T_N$ quasinilpotent. On the other hand, $0 \not \in \inter \, \sigma_{ap}(T)$ by Lemma \ref{lema1}. Obviously, $0 \not \in \inter \, \sigma_{LD}(T)$, and from Theorem \ref{Tequiv} it follows that $T$ is Koliha-Drazin invertible, and hence $0 \in \iso \, \sigma(T)$.
\end{proof}

Corollary \ref{cor-iso} may fail when we drop the assumption that $T \in L(X)$ admits a GKD or that $H_0(T)$ is closed.

\begin{example}
	{\em Let $T \in L(H)$, $H$ is a complex Hilbert space, be a compact normal operator such that $\sigma(T)$ is an infinite subset of $\mathbb{C}$. Since every compact operator is Riesz, $T$ does not admit a GKD (see  Remark 5.3 in \cite{caot}). Moreover, $H_0(T)=N(T)$ \cite[Example 3.9]{aienabook}, so $H_0(T)$ is closed. On the other hand, $0 \in \iso \, \sigma_{lw}(T)$, but $0 \in \acc \, \sigma(T)$.
		
		A standard way to construct such an operator is as follows. Let $H$ be a separable Hilbert space with scalar product $(\,,\,)$ and let $\{e_n:n \in \mathbb{N}\}$ be a total orthonormal system in $H$. If $(\lambda_n)$ is a sequence of complex numbers such that $(|\lambda_n|)$ is a decresing sequence converging to zero, then we define
		$$Tx=\sum_{n=1}^{\infty} \lambda_n(x, e_n)e_n \; \; \text{for all} \; \; x \in H.$$
		It is not too hard to check that $T$ is a compact normal operator and that $\sigma(T)=\{\lambda_n:n \in \mathbb{N}\} \cup \{0\}$.}
\end{example}

\begin{example}
	{\em Define $T=L \oplus O$, where $L$ and $O$ are the unilateral left shift and the zero operator on the Hilbert space $\ell_2(\mathbb{N})$, respectively. The spectrum and the surjective spectrum of $L$ are given by $\sigma(L)=\mathbb{D}$ and $\sigma_{su}(L)=\Gamma$, where $\mathbb{D}=\{\lambda \in \mathbb{C}:|\lambda| \leq 1\}$ and $\Gamma=\{\lambda \in \mathbb{C}:|\lambda|=1\}$. Since $L$ is surjective and $O$ is quasinilpotent, then $T$ admits a GKD. In addition, $H_0(T)$ is not closed (see Remark 2.6 in \cite{Salvatore}). We have
		\begin{equation}\label{jed-2}
		\sigma_{lw}(T) \subset \sigma_{su}(T)=\sigma_{su}(L) \cup \sigma_{su}(O)=\Gamma \cup \{0\}.
		\end{equation}
		Further, $\beta(T)=\beta(L)+\beta(O)=0+\infty=\infty$, and hence $0 \in \sigma_{lw}(T)$. From this and from \eqref{jed-2} we deduce that $0 \in \iso \, \sigma_{lw}(T)$. On the other hand, $\sigma(T)=\sigma(L) \cup \sigma(O)=\mathbb{D}$, and so $0 \in \acc \, \sigma(T)$.}
\end{example}

We obtain new characterizations of Drazin invertible operators if we consider the case where $T_N$ is nilpotent. 

\begin{theorem}\label{Tequiv1}
	Let $T \in L(X)$. The following conditions are equivalent:
	
	\smallskip
	
	\noindent {\rm (i)} $T$ is Drazin invertible;
	
	\smallskip
	
	\noindent {\rm (ii)} $0 \not \in \inter \, \sigma_{desc}(T)$ and there exists $(M,N) \in Red(T)$ such that $T_M$ is bounded below and $T_N$ is nilpotent;
	
	\smallskip
	
	\noindent {\rm (iii)} $0 \not \in \inter \, \sigma_{desc}(T)$ and there exists $(M,N) \in Red(T)$ such that $T_M$ is upper semi-Weyl and $T_N$ is nilpotent;
	
	\smallskip
	
	\noindent {\rm (iv)} $0 \not \in \inter \, \sigma_{LD}(T)$ and there exists $(M,N) \in Red(T)$ such that $T_M$ is surjective and $T_N$ is nilpotent;
	
	\smallskip
	
	\noindent {\rm (v)} $0 \not \in \inter \, \sigma_{LD}(T)$ and there exists $(M,N) \in Red(T)$ such that $T_M$ is lower semi-Weyl and $T_N$ is nilpotent.
\end{theorem}
\begin{proof}
	{\rm (i)} $\Longrightarrow$ {\rm (ii)}. There exists $(M,N) \in Red(T)$ such that $T_M$ is invertible and $T_N$ is nilpotent. Evidently, $T_M$ is bounded below and $0 \not \in \acc \, \sigma(T)$. From $\inter \, \sigma_{desc}(T) \subset \inter \, \sigma(T) \subset \acc \, \sigma(T)$ we conclude $0 \not \in \inter \, \sigma_{desc}(T)$.
	
	\smallskip
	
	\noindent {\rm (ii)} $\Longrightarrow$ {\rm (iii)}. Every bounded below operator is upper semi-Weyl.
	
	\smallskip
	
	\noindent {\rm (iii)} $\Longrightarrow$ {\rm (i)}. By \cite[Theorem 4.1]{caot}, $T$ is of Kato type and $0 \not \in \acc \, \sigma_{uw}(T)$. It means that there exists $(M_1, N_1) \in Red(T)$ such that $T_{M_1}$ is Kato and $T_{N_1}$ is nilpotent. As in the proof of Theorem \ref{Tequiv} we conclude that $0 \not \in \acc \, \sigma_{b}(T)$. Consequently, $0 \not \in \acc \, \sigma_{b}(T_{M_1})$. According to \cite[Proposition 3.2(iv)]{caot}, $T_{M_1}$ is invertible. As a consequence, $T=T_{M_1} \oplus T_{N_1}$, where $T_{M_1}$ is invertible and $T_{N_1}$ is nilpotent, and hence $T$ is Drazin invertible.
	
	\smallskip
	
	The statements {\rm (i)} $\Longrightarrow$  {\rm (iv)} $\Longrightarrow$ {\rm (v)} $\Longrightarrow$ {\rm (i)} can be proved similarly.
\end{proof}

\begin{remark}
	{\em Let $H$ be a complex Hilbert space and $T \in L(H)$. The Kato type spectrum of $T$, denoted by $\sigma_{Kt}(T)$, is the set of all $\lambda \in \mathbb{C}$ such that $T-\lambda I$ is not of Kato type. 
		
	\smallskip
	
	\noindent {\rm (i)}. An operator $T \in L(H)$ is left (right) Drazin invertible if and only if $T=T_M \oplus T_N$, where $T_M$ is bounded below (surjective) and $T_N$ is nilpotent \cite[Theorems 3.6 and 3.12]{berkani}. These facts and Theorem \ref{Tequiv1} imply
$$\sigma_D(T)=\inter \, \sigma_{desc}(T) \cup \sigma_{LD}(T)=\inter \, \sigma_{LD}(T) \cup \sigma_{RD}(T).$$		
In particular, suppose that the spectrum of $T$ is contained in a line. Then $\inter \, \sigma_{desc}(T)=\inter \, \sigma_{LD}(T)=\emptyset$. Moreover, $\sigma_D(T)=\sigma_{Kt}(T) \cup \inter \, \sigma(T)=\sigma_{Kt}(T)$ by \cite[Theorem 4.1]{caot}. It follows that
\begin{equation}\label{spectra}
\sigma_D(T)=\sigma_{LD}(T)=\sigma_{RD}(T)=\sigma_{Kt}(T).
\end{equation}	

\smallskip

\noindent {\rm (ii).} We recall that $T \in L(H)$ is called normal if $TT^{\star}=T^{\star}T$, where $T^{\star}$ is the Hilbert-adjoint of $T$.  If $T \in L(H)$ is normal, then $T$ satisfies \eqref{spectra}. Indeed, from \cite[Theorems 3.6 and 3.12]{berkani} we conclude that $\sigma_{Kt}(T) \subset \sigma_{LD}(T)$ and $\sigma_{Kt}(T) \subset \sigma_{RD}(T)$. Evidently, $\sigma_{LD}(T) \subset \sigma_D(T)$ and $\sigma_{RD}(T) \subset \sigma_D(T)$, and it is sufficient to show that $\sigma_D(T) \subset \sigma_{Kt}(T)$. Let $\lambda \not \in \sigma_{Kt}(T)$. Then $T-\lambda I=T_1 \oplus T_2$ with $T_1$ Kato and $T_2$ nilpotent. For sufficiently large $n \in \mathbb{N}$ we have $R((T-\lambda I)^n)=R(T_1^n)$, so $R((T-\lambda I)^n)$ is closed. Since $(T-\lambda I)^n$ is normal, then \cite[Theorem 6.2-G]{Taylor} gives 
\begin{equation*} \label{normal}
H=N((T-\lambda I)^n) \oplus R((T-\lambda I)^n).
\end{equation*}		
Now, \cite[Theorem 3.6]{aienabook} implies that $p(T-\lambda I)=q(T-\lambda I) \leq n$, and hence $T-\lambda I$ is Drazin invertible.}
\end{remark}

\section{Browder type theorems}

We recall that $T \in L(X)$ is said to satisfy Browder\textquoteright s theorem if $\sigma_w(T)=\sigma_b(T)$.

\begin{corollary}\label{cor-eq}
	Let $T \in L(X)$. The following statements are equivalent:
	
	\smallskip
	
	\noindent {\rm (i)} Browder\textquoteright s theorem holds for $T$;
	
	\smallskip
	
	\noindent {\rm (ii)} $\inter \, \sigma_{desc}(T) \subset \sigma_w(T)$;
	
	\smallskip
	
	\noindent {\rm (iii)} $\inter \, \sigma_{su}(T) \subset \sigma_w(T)$;
	
	\smallskip
	
	\noindent {\rm (iv)} $\inter \, \sigma_{ap}(T) \subset \sigma_w(T)$.
	
	\smallskip
	
	In particular, if either $\sigma_{ap}(T)$ or $\sigma_{su}(T)$ has no interior points then Browder\textquoteright s theorem holds for $T$.
\end{corollary}
\begin{proof}
	{\rm (i)} $\Longleftrightarrow$ {\rm (ii)}. Since $\sigma_{uw}(T) \subset \sigma_w(T) \subset \sigma_b(T)$, $\sigma_b(T)=\sigma_{w}(T) \cup \inter \, \sigma_{desc}(T)$ by Theorem \ref{Theorem_1}, part {\rm (vii)}. Now, the result follows immediately.
	
	\smallskip
	
	\noindent {\rm (ii)} $\Longleftrightarrow$ {\rm (iii)}. Follows from $\inter \, \sigma_{desc}(T)=\inter \, \sigma_{su}(T)$ \cite[Corollary 2.18 (5)]{TUD}.
	
	\smallskip
	
	\noindent {\rm (i)} $\Longleftrightarrow$ {\rm (iv)}. Observe that $\sigma_b(T)=\sigma_w(T) \cup \inter \, \sigma_{ap}(T)$ by Theorem \ref{Theorem_1}, part (viii), and \cite[Corollary 2.18 (5)]{TUD}.
	
	\smallskip
	
	To prove the last assertion, apply {\rm (iii)} $\Longrightarrow$ {\rm (i)} or {\rm (iv)} $\Longrightarrow$ {\rm (i)}.
\end{proof}

An operator $T \in L(X)$ satisfies  a-Browder\textquoteright s   theorem if $\sigma_{uw}(T)=\sigma_{ub}(T)$.

\begin{theorem}\label{cor-aBrowder}
	Let $T \in L(X)$. The following conditions are equivalent:
	
	\smallskip
	
	\noindent {\rm (i)} $\inter \, \sigma_{desc}(T) \subset \sigma_{uw}(T)$;
	
	\smallskip
	
	\noindent {\rm (ii)} $T^{\prime}$ has {\rm SVEP} at every $\lambda \not \in \sigma_{uw}(T)$;
	
	\smallskip
	
	\noindent {\rm (iii)} $\sigma_{desc}(T) \subset \sigma_{uw}(T)$;
	
	\smallskip
	
	\noindent {\rm (iv)} a-Browder\textquoteright s theorem holds for $T$ and $\sigma(T)=\sigma_{ap}(T)$.
\end{theorem}
\begin{proof}
	{\rm (i)} $\Longrightarrow$ {\rm (ii)}. Let $\inter \, \sigma_{desc}(T) \subset \sigma_{uw}(T)$ and $\lambda \not \in \sigma_{uw}(T)$. From Theorem \ref{Theorem_1}, part (vii), we have that $T-\lambda I$ is Browder and thus $\lambda \not \in \acc \, \sigma(T)$ by \cite[Corollary 20.20]{Muller}. Consequently, $T^{\prime}$ has {\rm SVEP} at $\lambda$.
	
	\smallskip
	
	\noindent {\rm (ii)} $\Longrightarrow$ {\rm (iii)}. Let $\lambda \not \in \sigma_{uw}(T)$. Then $T-\lambda I$ is upper semi-Fredholm and \cite[Theorem 3.17]{aienabook} ensures that $q(T-\lambda I)<\infty$. It follows that $\lambda \not \in \sigma_{desc}(T)$.
	
	\smallskip
	
	\noindent {\rm (iii)} $\Longrightarrow$ {\rm (i)}. Obvious.
	
	\smallskip
	
	\noindent {\rm (i)} $\Longrightarrow$ {\rm (iv)}. From $\inter \, \sigma_{desc}(T) \subset \sigma_{uw}(T) \subset \sigma_{ap}(T)$ and from part {\rm (v)} of Theorem \ref{Theorem_1} we obtain $\sigma(T)=\sigma_{ap}(T)$.
	Further, $\sigma_b(T)=\sigma_{uw}(T)$ by Theorem \ref{Theorem_1}, part {\rm (vii)}. From this and from $\sigma_{uw}(T) \subset \sigma_{ub}(T) \subset \sigma_b(T)$ we deduce $\sigma_{uw}(T)=\sigma_{ub}(T)$, so a-Browder\textquoteright s theorem holds for $T$.
	
	\smallskip
	
	\noindent {\rm (iv)} $\Longrightarrow$ {\rm (i)}. Since $\sigma_{ub}(T)=\sigma_{uw}(T)$, $\acc \, \sigma_{ap}(T) \subset \sigma_{uw}(T)$ by \cite[Theorem 3.65, part (v)]{aienabook}. Now, we have $\inter \, \sigma_{desc}(T) \subset \acc \, \sigma(T)=\acc \, \sigma_{ap}(T) \subset \sigma_{uw}(T)$. 
\end{proof}
\noindent The implication
\begin{equation}\label{reversed}
T^{\prime} \, \text{has SVEP at every} \, \lambda \not \in \sigma_{uw}(T) \Rightarrow \text{a-Browder\textquoteright s theorem holds for} \, T
\end{equation}
was proved in \cite[Theorem 2.3]{aienaTstar}. Moreover, it was also observed in \cite{aienaTstar} that the implication \eqref{reversed} can not be reversed in general. From Theorem \ref{cor-aBrowder} we may obtain more information concerning this issue. Namely, from the equivalence {\rm (ii)} $\Longleftrightarrow$ {\rm (iv)} of Theorem \ref{cor-aBrowder} we conclude that the implication \eqref{reversed} can be reversed if $\sigma(T)=\sigma_{ap}(T)$. If $\sigma_{ap}(T)$ is a proper subset of $\sigma(T)$, then the converse of the assertion \eqref{reversed} does not hold. We now apply this observations to weighted right shifts and isometries. For relevant background material concerning weighted right shifts, see \cite[pp. 85-90]{LN}.

\begin{example}
	{\em A weighted right shift $T$ on $\ell_2(\mathbb{N})$ is defined by
		$$T(x_1, x_2, \cdots)=(0, w_1x_1, w_2x_2, \cdots) \; \; \text{for all} \; \; (x_1, x_2, \cdots) \in \ell_2(\mathbb{N}),$$
		where $(w_n)$ is a bounded sequence of strictly positive numbers ($0<w_n\leq 1$ for all $n \in \mathbb{N}$). It is easily seen that $T$ is a bounded linear operator on $\ell_2(\mathbb{N})$ and that $T$ has no eigenvalues. It follows that $T$ has SVEP, so every weighted right shift satisfies a-Browder\textquoteright s theorem by \cite[Theorem 2.3 (i)]{aienaTstar}.
		
		The approximate point spectrum and the spectrum of $T$ are described as follows:
		$$\sigma_{ap}(T)=\{\lambda \in \mathbb{C}:i(T)\leq|\lambda| \leq r(T)\}, \; \sigma(T)=\{\lambda \in \mathbb{C}:|\lambda| \leq r(T)\},$$
		where $r(T)=\max\{|\lambda|: \lambda \in \sigma(T)\}=\displaystyle \lim_{n \to \infty} \|T^n\|^{1/n}$ is the spectral radius of $T$ and $i(T)=\displaystyle \lim_{n \to \infty} k(T^n)^{1/n}$. 
		
		It is possible to find a weight sequence $(w_n)$ for which the corresponding weighted right shift $T$ satisfies $0<i(T)<r(T)$. P. Aiena et al. proved \cite[Example 2.5]{aienaTstar} that $0 \not \in \sigma_{uw}(T)$ and that $T^{\prime}$ does not have SVEP at $0$, so the converse of \eqref{reversed} does not hold for $T$. On the other hand, using Theorem \ref{cor-aBrowder} we obtain the same conclusion in a rather direct way if we observe that $\sigma_{ap}(T) \subsetneq \sigma(T)$.}
\end{example}

\begin{example}
	{\em Let $T \in L(X)$ be a non-invertible isometry. It is known that $\sigma(T)=\{\lambda \in \mathbb{C}:|\lambda| \leq 1\}$ and $\sigma_{ap}(T)=\{\lambda \in \mathbb{C}:|\lambda|=1\}$ \cite[p. 80]{LN}. Consequently, $\inter \, \sigma_{ap}(T)$ is empty set, so $T$ has SVEP, and from \cite[Corollary 2.4]{aienaTstar} we infer that a-Browder\textquoteright s theorem holds for $T$. Since $\sigma_{ap}(T)$ is a proper subset of $\sigma(T)$, the reverse of the assertion \eqref{reversed} does not hold for $T$.}
\end{example}

\begin{theorem}\label{converseprime}
	Let $T \in L(X)$. The following conditions are equivalent:
	
	\smallskip
	
	\noindent {\rm (i)} $\inter \, \sigma_{LD}(T) \subset \sigma_{lw}(T)$;
	
	\smallskip
	
	\noindent {\rm (ii)} $T$ has {\rm SVEP} at every $\lambda \not \in \sigma_{lw}(T)$;
	
	\smallskip
	
	\noindent {\rm (iii)} $\sigma_{LD}(T) \subset \sigma_{lw}(T)$;
	
	\smallskip
	
	\noindent {\rm (iv)} a-Browder\textquoteright s theorem holds for $T^{\prime}$ and $\sigma(T)=\sigma_{su}(T)$.
\end{theorem}
\begin{proof}
	{\rm (i)} $\Longrightarrow$  {\rm (ii)}. Since $\inter \, \sigma_{LD}(T) \subset \sigma_{lw}(T)$, $\sigma_b(T)=\sigma_{lw}(T)$ by Theorem \ref{Theorem_1}, part (viii). If $\lambda \not \in \sigma_{lw}(T)$, then $\lambda \not \in \sigma_b(T)$, and hence $\lambda \not \in \acc \, \sigma(T)$ \cite[Corollary 20.20]{Muller}. Consequently, $T$ has SVEP at $\lambda$.
	
	\smallskip
	
	\noindent {\rm (ii)} $\Longrightarrow$  {\rm (iii)}. Let $\lambda \not \in \sigma_{lw}(T)$. Then $T-\lambda I$ is lower semi-Fredholm, so $T-\lambda I$ is of Kato type. By \cite[Theorem 3.16]{aienabook} the SVEP of $T$ at $\lambda$ implies that $p=p(T-\lambda)<\infty$. In addition, $R((T-\lambda I)^{p+1})$ is closed, so $\lambda \not \in \sigma_{LD}(T)$.
	
	\smallskip
	
	\noindent {\rm (iii)} $\Longrightarrow$  {\rm (i)}. It is clear.
	
	\smallskip
	
	\noindent {\rm (i)} $\Longrightarrow$  {\rm (iv)}. Since  $\inter \, \sigma_{LD}(T) \subset \sigma_{lw}(T) \subset \sigma_{su}(T)$ we conclude by Theorem \ref{Theorem_1}, parts (vi) and (viii), that $\sigma(T)=\sigma_{su}(T)$ and $\sigma_b(T)=\sigma_{lw}(T)$. Now, $\sigma_b(T^{\prime})=\sigma_{uw}(T^{\prime}) \subset \sigma_{ub}(T^{\prime}) \subset \sigma_{b}(T^{\prime})$. It follows that $\sigma_{uw}(T^{\prime})=\sigma_{ub}(T^{\prime})$, so a-Browder\textquoteright s theorem holds for $T^{\prime}$.
	
	\smallskip
	
	\noindent {\rm (iv)} $\Longrightarrow$  {\rm (i)}. Since $\sigma_{uw}(T^{\prime})=\sigma_{ub}(T^{\prime})$, $\sigma_{lw}(T)=\sigma_{lb}(T)$ and therefore \cite[Theorem 3.65, part (v)]{aienabook} ensures that $\acc \, \sigma_{su}(T) \subset \sigma_{lw}(T)$. It follows that $\inter \, \sigma_{LD}(T) \subset \acc \, \sigma(T)=\acc \, \sigma_{su}(T) \subset \sigma_{lw}(T)$.
\end{proof}

The implication 
\begin{equation}\label{reverseprime}
T \, \text{has SVEP at every} \, \lambda \not \in \sigma_{lw}(T) \Rightarrow \text{a-Browder\textquoteright s theorem holds for} \, T^{\prime}
\end{equation}
can be found in \cite[Theorem 2.3]{aienaTstar}. What is more, the converse of the assertion \eqref{reverseprime} does not hold in general \cite[Example 2.5]{aienaTstar}. According to Theorem \ref{converseprime}, the implication \eqref{reverseprime} can be reversed if $\sigma(T)=\sigma_{su}(T)$.

\medskip 

\noindent \author{Milo\v s D. Cvetkovi\'c}

\noindent{Faculty of Occupational Safety\\University of Ni\v s\\
\v{C}arnojevi\'ca 10a, 18000 Ni\v s, Serbia}

\noindent {\it E-mail}: {\tt miloscvetkovic83@gmail.com}

\bigskip

\noindent \author{Sne\v zana \v C. \v Zivkovi\'c-Zlatanovi\'c}

\noindent{Faculty of Sciences and Mathematics\\University of Ni\v s\\
Vi\v{s}egradska 33, P.O. Box 224 \\ 18000 Ni\v s, Serbia}

\noindent {\it E-mail}: {\tt mladvlad@mts.rs}

\end{document}